\documentclass[11pt]{article}

\usepackage{amsmath,amsfonts,amsbsy,amstext,amsopn,amssymb}

\usepackage{longtable}
\usepackage{theorem}
\usepackage{algorithmic}
\usepackage{algorithm}
\usepackage{color}
\usepackage{subfig} 
\usepackage{graphicx}
\usepackage{lscape}

\makeatletter \@addtoreset{equation}{section}

\makeatother

{\theorembodyfont{\slshape}
\newtheorem{theorem}{Theorem}

\newtheorem{lemma}{Lemma}
\newtheorem{corollary}{Corollary}

}
{\theorembodyfont{\rmfamily}
\newtheorem{definition}{Definition}

}

\newcommand{\vect}{{\sf{vec}}}

\newcommand{\rt}{{\top }}

\def\bfe{{\mathbf e}}

\newcommand{\qed}{$\hfill{\Box}$} 

\def\rel{{\rm rel}}
\def\u{{\rm U}}

\def\bfe{{\mathbf e}}

\def\ILS{{\mathrm {ILS}}}

\usepackage{multirow}

\newcommand {\eq} [1] {\begin{equation}\label{#1}}
\newcommand {\en} {\end{equation}}
\newcommand {\R}        {{\mathbb R}}

\newcommand {\mat}      [1] {\left[\begin{array}{#1}}
\newcommand {\rix}          {\end{array}\right]}

\newcommand {\trace}    {\mathop{\rm Tr}\nolimits}

 \catcode`@=11
 \font\tenex=cmex10 
 \newdimen\p@renwd
 \setbox0=\hbox{\tenex B} \p@renwd=\wd0 
 \def\bmat#1{\begingroup \m@th
   \setbox\z@\vbox{\def\cr{\crcr\noalign{\kern2\p@\global\let\cr\endline}}%
     \ialign{$##$\hfil\kern2\p@\kern\p@renwd&\thinspace\hfil$##$\hfil
       &&\quad\hfil$##$\hfil\crcr
       \omit\strut\hfil\crcr\noalign{\kern-\baselineskip}%
       #1\crcr\omit\strut\cr}}%
   \setbox\tw@\vbox{\unvcopy\z@\global\setbox\@ne\lastbox}%
   \setbox\tw@\hbox{\unhbox\@ne\unskip\global\setbox\@ne\lastbox}%
   \setbox\tw@\hbox{$\kern\wd\@ne\kern-\p@renwd\left[\kern-\wd\@ne
     \global\setbox\@ne\vbox{\box\@ne\kern2\p@}%
     \vcenter{\kern-\ht\@ne\unvbox\z@\kern-\baselineskip}\,\right]$}%
   \null\;\vbox{\kern\ht\@ne\box\tw@}\endgroup}
 \catcode`@=12
\newcommand*{\affaddr}[1]{#1} 
\newcommand*{\affmark}[1][*]{\textsuperscript{#1}}

\begin{document}

\title{\bf Backward error  and condition number analysis for the indefinite linear least squares problem}
\author{Huai-An Diao\affmark[1]
\thanks{hadiao@nenu.edu.cn, hadiao78@yahoo.com}\\
\and
Tong-Yu Zhou\affmark[1]\affmark[2]
\thanks{
183051056@qq.com}\\
\affaddr{\affmark[1]School of Mathematics and Statistics,
Northeast Normal University, \\
No. 5268 Renmin Street, Chang Chun 130024,
P.R. China.}\\
\affaddr{\affmark[2]Current address: Shenyang No. 108 Middle School, \\
Taiyuan North Road No. 14,
Shenyang  110001, P.R. China.}\\
}
\date{}
\maketitle

\begin{quote}
{\bf Abstract.} In this paper,
we concentrate on the backward error and condition number of the indefinite least squares problem. For the normwise backward error  of the indefinite least square problem, we adopt the linearization method to derive the tight estimations for the exact normwise backward  errors. Using the dual techniques of condition number theory \cite{22.0}, we derive the explicit expressions of the mixed and componentwise condition numbers
 for  the linear function of the solution for the indefinite least squares problem. The tight upper bounds for the derived  mixed and componentwise condition numbers are obtained, which can be estimated efficiently by means of the classical power method for estimating matrix 1-norm \cite[Chapter 15]{Higham2002Book}  during using the QR-Cholesky method \cite{1.0} for solving the indefinite least squares problem. The numerical examples show that the derived condition numbers can give sharp perturbation bound with respect to the interested component of the solution. And the linearization  estimations are effective for the normwise backward errors.
\end{quote}

{\small {\bf Keywords:}
Indefinite least squares;~normwise backward error;~linearization estimate;~componentwise perturbation;~condition number;  ~adjoint operator;~power method.\\

}


\section{Introduction}
\setcounter{equation}{0}

The indefinite least squares (ILS) problem has  the following form:
\begin{equation}\label{eq:ILS}
 \mbox{ILS}:\qquad\min_{x}(b-Ax)^\top\Sigma_{pq}(b-Ax),
\end{equation}
where $ A\in
\R^{m\times n},\, b\in
\R^{m},\, m\geq n $ and the signature matrix
\begin{equation}\label{eq:J}
\Sigma_{pq}=\left( \begin{matrix}
I_{p}&0\\0&-I_{q} \end{matrix} \right),\quad
 p+q=m.
\end{equation}
When $p = 0 $ or $ q = 0 $,  \eqref{eq:ILS} reduces to the linear least squares (LS) problem and the quadratic
form is definite, while for $ pq > 0 $, and then \eqref{eq:ILS} is to minimize a genuinely indefinite
quadratic form. The ILS comes from the total least squares problem \cite{30.0} and $H^\infty$-smoothing in  optimization; see \cite{Hassibi96linearestimation,Sayed96inertiaproperties}  and   references therein. From the first order optimality condition of \eqref{eq:ILS}, the normal equations for \eqref{eq:ILS} is
\begin{equation}\label{eq:normal}
A^\top\Sigma_{pq}Ax=A^\top b.
\end{equation}
Because the Hessian matrix of the quadratic to be minimized in \eqref{eq:ILS} is $2A^\top\Sigma_{pq}A$, it follows that the ILS problem has a unique solution if and only if
\begin{equation}\label{eq:unique cond}
A^\top \Sigma_{pq}A \mbox{ is postive defenite}.
\end{equation}
We will assume throughout this paper that \eqref{eq:unique cond} holds. Note that \eqref{eq:unique cond} implies that $p\geq n$ and $A(1:p,\, 1:n)$ (and hence $A$) has full rank. For a genuinely  ILS we therefore need $m>n$.

For the numerical algorithms for solving ILS, there are two types: Direct method and iterative method. For the direct method which is suitable for small or medium
scale ILS, Chandrasekaran, Gu and Sayed \cite{1.0} proposed QR-Cholesky method, then later a direct method based on the hyperbolic QR was introduced in \cite{5.0}, which was improved by Xu \cite{Xu2004}. Recently Mastronardi and Van Dooren \cite{MastronardiVanDooren2014BIT} proposed a direct solver based on orthogonal transformation of an indefinite symmetric matrix to a matrix in a proper block anti-triangular form. For large scale problem, Liu et al. \cite{Liu2011,LiuZhang2013ILS} introduced preconditioned conjugate gradient method (PCG). Then the block SOR methods was considered for ILS \cite{LiuLiuCalcolo2014}.

In matrix computations, sensitivity analysis is important. Condition number describes the worst case sensitivity of the output data with respect to the perturbations on the input data, see \cite{Cucker2013Book} and references therein. When a problem with large condition numbers, it is usually called {\it ill-posed} \cite{Demmel1987NumerMath}, which means that we cannot trust the computed solution. Backward error is the smallest perturbation of the input data to make the computed solution to be the exact solution of the corresponding perturbed problem mathematically \cite{Higham2002Book}. With backward error and condition number, the forward error of the problem can be bounded by the following rule of thumb \cite[Page 9]{Higham2002Book}
$$
\mbox{ forward error } \lesssim \mbox{ condition number } \times \mbox{ backward error,}
$$
where the inequality, with errors of higher order terms, can be sharp.

For perturbation analysis, normwise perturbation analysis is classical, and measures the perturbations by their norms. The {\it normwise} condition number for a given problem was defined by Rice \cite{Rice} and a general theory of it was established. The normwise condition number measures the input and output data errors \cite{Higham2002Book} using their  norms. However, when the data is badly scaled or sparse, normwise perturbation bounds allow large relative perturbations on small entries and may give over-estimated bounds. Instead of measuring perturbations by norms,   Skeel in \cite{Skeel79} used {\em componentwise} perturbation analysis to investigate the stability  of Gaussian elimination for linear systems.  Since then, the componentwise perturbation analysis has received considerable attentions for many classical problems in numerical linear algebra; see the comprehensive survey \cite{MR1314843} and the references therein. Componentwise perturbation analysis is more suitable since it measures perturbation errors for each component of the input data. For example, in a modern computer, the floating point system is designed to store a real number in the computer. Every stored number should be rounded to the nearest floating point number, which means that we should measure errors for each component of the input data. In componentwise analysis, two types of condition numbers, described as mixed and componentwise were proposed; see \cite{CuckerDiaoWei2007,Gohberg} for details.

The concept of backward error can be tracked to Wilkinson and others; see  \cite[Page 33]{Higham2002Book} for details. Backward error analysis can be used to examine the stability of numerical algorithms in matrix computation. Moreover, backward error can serve as the basis of effective stopping criteria for the iterative method for large scale problems. There have been many works on the backward error analysis for the linear least squares problem \cite{Stewart,Higham,Walden,Karlson,Grcar03optimalsensitivity,Grcar07estimatesof}, the scale total least squares (STLS) problem \cite{Chang}, and the equality constrained least squares (LSE) problem and the least squares problem over a sphere (LSS)~\cite{Cox,Malyshev}. Because the formulae for and bounds on backward errors for least squares problems are expensive to evaluate, the linearization estimate for them was proposed; see  for \cite{Chang,Grcar03optimalsensitivity,HighamHigham1992BackCondStr,LiuXinguoNLAA2012} and references therein. To our best knowledge, there are no works on the normwise backward error for ILS. In this paper, we will introduce the normwise backward error for ILS and derive its linearization estimate.

Let us review some previous works on the perturbations analysis for ILS. For normwise perturbation analysis, we refer to the papers \cite{5.0,Grcar2011,Wang2009ILS} and references therein. Li et al. \cite{LiWangYang2014} considered componentwise perturbation analysis for the solution of ILS and obtained the explicit expressions for the mixed and componentwise condition numbers. In 2015,  Zhou  \cite{Zhou2015} firstly introduced the condition numbers for a linear function of the solution for ILS, and the corresponding condition numbers expressions were derived based on the dual techniques  \cite{22.0}. Also the linearization estimate for the normwise backward error was obtained \cite{Zhou2015}. In 2016, the condition numbers for a linear function of the solution for ILS also named as \lq\lq partial condition numbers for ILS \rq\rq  \cite{LiWang2016} is studied. The explicit expressions for these condition numbers are derived. Also the probabilistic spectral norm estimator and the small-sample statistical condition estimation method are proposed to estimate condition numbers  \cite{LiWang2016}, but the authors do not take account of the numerical method for ILS to reduce the computational complexity of condition estimations. Usually, in the field of condition estimations \cite[Chapter 15]{Higham2002Book} in numerical linear algebra, how to incorporate condition estimations to the numerical method  by utilizing the already computed matrix decompositions is crucial, thus the computational complexity during condition estimations can be reduced when the already computed matrix decompositions are used to devise the algorithms for condition estimations.

In this paper, we study the sensitivity  of a linear function of the ILS solution $x$ to perturbations on the date $A$  and $ b$, which is defined as
\begin{equation}\label{eq:g dfn}
g(A,b)=L^\top(A^\top\Sigma_{pq}A)^{-1}A^\top\Sigma_{pq}b,
\end{equation}
where $L$ is an $n$-by-$k$, $k \le n$, matrix introduced for
the selection of the solution components. For example, when
$L = I_n$ ($k=n$), all the $n$ components of the solution $x$
are equally selected. When $L = e_i$ ($k=1$), the $i$th unit
vector in $\R^n$, then only the $i$th component of the
solution is selected. In the reminder of this paper, we always suppose that $L$ is not numerically perturbed.

One objective of this paper is to obtain the explicit expressions for componentwise and mixed condition numbers of the linear function of the solution when perturbations on data are measured componentwise and the perturbations on the solution are measured either componentwise or normwise by means of the dual techniques \cite{22.0}. In particular, as also mentioned in \cite{22.0}, the dual techniques enable us to derive condition numbers by maximizing a linear function over a space of smaller dimension than the data space. The methodology to deduce condition numbers expressions is different with that in ~\cite{LiWang2016}. And tight upper bounds also are obtained, which can be estimated efficiently via the classical power method \cite[Chapter 15]{Higham2002Book} during using the QR-Cholesky method \cite{1.0} for solving ILS by means of utilizing the already computed matrix decompositions to reduce the computational complexity of condition estimations. We should point out that the proposed condition estimations are deterministic, which is different with probabilistic or statistical condition estimations ~\cite{LiWang2016}.

Another objective of this paper is to  introduce the normwis backward error for ILS. To our best knowledge, there have been no papers on this subject. Because the normwis backward error for ILS is a nonlinear optimization problem, which usually is not easy to derive its explicit expression, we study the linearization estimate for it and the explicit expression for the linearization estimate is obtained in Section \ref{sec:backward error}.

The paper is organized as follows. We firstly derive the linearization estimate of the normwise backward error for ILS in Section \ref{sec:backward error}. In Section \ref{sec:cond}, the dual techniques for deriving condition number  \cite{22.0} is reviewed and applied to ILS, then we propose the power method \cite[Chapter 15]{Higham2002Book}  to estimate tight bounds for the mixed and componentwise condition numbers for the linear function of the solution of ILS by taking account of the already computed matrix decompositions during solving ILS by means of the QR-Cholesky method \cite{1.0}. We do some numerical examples to show the effectiveness of the proposed condition numbers and linearization estimate for the normwise backward error in Section \ref{sec:nume ex}. At end, in Section \ref{sec:con} concluding remarks are drawn.

\section{Linearization estimate of the normwise backward error for ILS}\label{sec:backward error}

In this section we will study the linearization estimate for the normwise backward error for ILS. Suppose $y$ is an approximate solution of \eqref{eq:ILS} . we consider the set of perturbations
\begin{equation}\label{eq:xi}
\xi_{\ILS}=\{(\Delta A,\, \Delta b)|(A+\Delta A)^\top\Sigma_{pq}(b+\Delta b-(A+\Delta A)y)=0\}
\end{equation}
We consider the normwise backward error of $y$ defined as follows:
\begin{equation}\label{eq:xhwu}
\mu=\min_{(\Delta A,\Delta b)\in\xi_{\ILS}}\|[\Delta A,\, \theta \Delta b]\|_{F},
\end{equation}
where $\|\cdot\|_F$ is Frobenius norm, and $\theta$  is a positive parameter to balance the norm weight between $\Delta A$ and $\Delta b$. It is not difficult to see that the perturbation equation in $\xi_{\ILS} $ can be rewritten as
\begin{equation}\label{eq:rewritten}
J_{\ILS}\begin{pmatrix}
\vect(\Delta A)\\\theta \Delta b \end{pmatrix}=-A^\top\Sigma_{pq}r_y+\Delta A^\top\Sigma_{pq}(\Delta Ay-\Delta b),
  \end{equation}
where
\[
J_{\ILS}=\left(I_{n} \otimes r_y^\top \Sigma_{pq}- A^\top\Sigma_{pq}(y^\top \otimes I _{m}),\, \theta^{-1}A^\top\Sigma_{pq}\right),
\]
$r_y=b-Ay$ and $\vect(A)$ denotes the vector obtained by stacking
the columns of a matrix $A$,  $\otimes$ is the Kronecker
product operator \cite{Graham81}.

Since $\mu$ is the minimizer of a nonlinear optimization problem and in general it is not easy to derive its explicit expression. We will introduce $\bar{\mu}$ of $\mu$ to estimate it, which is a linearization estimate of $\mu$ defined as
\begin{equation}\label{eq:xhwugj}
 \bar{\mu}=\|J_{\ILS}^{\dagger}A^\top\Sigma_{pq}r_y\|_2,
\end{equation}
where $\|\cdot\|_2$ is 2-norm of a vector or spectral norm of a matrix. If $r_y\neq 0$, then $J_{\ILS} $ has full row rank. Let
\[
\psi(\Delta A,\Delta b)= \Delta A^\top\Sigma_{pq}(\Delta Ay-\Delta b)=\Delta A^\top\Sigma_{pq} (\Delta A,\theta \Delta b)\left(\begin{matrix}
y\\-\theta^{-1}\end{matrix}\right).
 \]
We have
 \begin{align}\label{eq:psi ef}
 \|\psi(\Delta A,\Delta A)\|_2 &\leq\|\Delta A^\top\Sigma_{pq}\|_2\|(\Delta A,\theta \Delta b)\|_F\sqrt{\theta^{-2}+\|y\|_2^{2}}\\
  &\leq\|\Delta A^\top\|_2\|[\Delta A,\theta \Delta b]\|_F\sqrt{\theta^{-2}+\|y\|_2^{2}} \nonumber \\
  &\leq\|[\Delta A,\theta \Delta b ]\|_{F}^{2}\eta_{1}, \nonumber
  \end{align}
 where $\eta_{1}=\sqrt{\theta^{-2}+\|y\|^{2}}$. The following theorem shows that $ \bar{\mu}$ is a sharp estimate of  $\mu$, provided that  $\bar{\mu}$  is sufficiently small.

\begin{theorem}
If  $ 4\eta_{1}\|J_{\ILS}^{\dagger}\|\bar{\mu}<1$, then
\[
\frac{2}{1+\sqrt{2}}\,\bar{\mu} \leq\mu \leq 2\bar{\mu},
\]
where $\eta_{1}=\sqrt{\theta^{-2}+\|y\|^{2}}.$
\end{theorem}
\begin{proof}
Since $J_{\ILS}$  has full row rank, premultiplying both sides of \eqref{eq:rewritten}  by $J_{\ILS}^{\dagger} $  gives
\begin{equation}
J_{\ILS}^{\dag}J_{ILS}\begin{pmatrix}
\vect(\Delta A)\\\theta \Delta b \end{pmatrix} =-J_{\ILS}^{\dag}A^\top\Sigma_{pq}r_y+J_{\ILS}^{\dag}\Delta A^\top\Sigma_{pq}(\Delta Ay-\Delta b).
 \end{equation}
This combined with \eqref{eq:psi ef}  implies that
\begin{align*}
\bar{\mu}&\leq \|[\Delta A,\theta \Delta A]\|_{F}+\|J^{\dag}\psi(\Delta A,\Delta b)\|_2\leq\|[\Delta A,\theta \Delta b]\|_{F}+\eta_{1}\|J_{\ILS}^{\dagger}\|\|[\Delta A,\theta \Delta b]\|_{F}^{2}.
\end{align*}
From this inequality and applying the assumption, we obtain that
\[
\mu \geq \frac{2}{1+\sqrt{2}}\bar{\mu}.
\]
In order to prove $\mu \leq 2\bar{\mu}$, we consider the following nonlinear system:
\begin{equation}\label{eq:written}
 \left(\begin{matrix}
\vect(\Delta A)\\\theta \Delta b \end{matrix}\right )=-J_{\ILS}^{\dagger}A^\top\Sigma_{pq}r_y+J_{\ILS}^{\dagger}\Delta A^\top\Sigma_{pq}(\Delta Ay-\Delta b).
\end{equation}
Note that $J_{\ILS}J_{\ILS}^{\dagger}=I$, which tells us that any solution of  \eqref{eq:written}  is also a solution of  \eqref{eq:rewritten} .
Consider the following mapping:
\[
\Gamma:\left(\begin{matrix}
\vect(\Delta A)\\\theta \Delta b \end{matrix} \right)\rightarrow -J_{\ILS}^{\dagger}A^\top\Sigma_{pq}r_y+J_{\ILS}^{\dagger}\Delta A^\top\Sigma_{pq}(\Delta Ay-\Delta b).
\]
It is not difficult to see that $ \Gamma $ is a continuous mapping from  $ \R^{mn+m}$ to $\R^{mn+m}$. Let
\begin{align*}
\xi_{*}&=\frac{2\bar{\mu}}{1+\sqrt{1-4\eta_{1}\|J_{\ILS}^{\dagger}\|\bar{\mu}}},\\
S&=\left\{\left(\begin{matrix}
\vect(\Delta A)\\\theta \Delta b \end{matrix}\right )\mid \Delta A\in \R^{m\times n},\, \Delta b\in \R^{m},\|[\Delta A,\, \theta \Delta b]\|_{F}\leq \xi_{*}\right\}.
\end{align*}
Obviously, the set $ S $ is bounded, closed and convex. Furthermore, for any
$$
\left(\begin{matrix}
\vect(\Delta A)\\\theta \Delta b \end{matrix}\right )\in S
$$
we can deduce that
\[
\left\|\Gamma\left(\begin{matrix}
\vect(\Delta A)\\\theta \Delta b \end{matrix} \right)\right\|_2\leq\bar{\mu} +\eta_{1} \|J_{\ILS}^{\dagger}\|\left\|\left(\begin{matrix}
\vect(\Delta A)\\\theta \Delta b \end{matrix} \right)\right\|_2^{2}\leq\xi_{*}.
\]
By the Brouwer fixed-point theorem, the mapping  $\Gamma $  has  a fixed point
$$
\begin{pmatrix}
\vect(\Delta A_{*})\\\theta \Delta b_{*} \end{pmatrix} \in S,
$$
and thus the system \eqref{eq:written} has a solution in S. Hence,
\[
\mu \leq \left\|\begin{pmatrix}
\vect(\Delta A_{*})\\\theta \Delta b_{*}\end{pmatrix}\right\|_2\leq \xi_{*}<2\bar{\mu}.
\]
\qed 
 \end{proof}

\section{Condition numbers for ILS}\label{sec:cond}

In this section we will derive the explicit condition numbers expressions for a linear function of the solution of ILS by means of the dual techniques \cite{22.0} under componentwise perturbations. Also tight upper bounds for the mixed and componentwise condition numbers are obtained, which can be estimated efficiently through the classical power method \cite[Chapter 15]{Higham2002Book}  by taking account of the already computed decomposition of the QR-Cholesky method \cite{1.0} for solving ILS.

\subsection{Dual techniques}

For the Euclidean spaces $E$  and  $G$  equipped scalar products $\langle \cdot ,\cdot \rangle_{E}$ and  $\langle\cdot,\cdot \rangle_{G}$ respectively, let a linear operator $ J:\, E \rightarrow G $ be well defined. Denote the corresponding norm norms $\|\cdot\|_{E} $ and $\|\cdot\|_{G}$ respectively. The well-known adjoint operator and dual norm are defined as follows.

\begin{definition}\label{def:adj dual}
The adjoint operator of  $J$,\, $J^{\ast}:G \rightarrow E $  is defined by
\[
\langle y,Jx \rangle_{G}=\langle J^{\ast}y,x\rangle_{E}
\]
\end{definition}
where $x\in E $ and $y \in G $.  The dual norm $ \|\cdot\|_{E^{\ast}}$  of  $ \|\cdot\|_{E} $  is defined by
\[
\|x\|_{E{\ast}}=\max_{u\neq0}\frac{ \langle  x,u \rangle _{E}}{\|u\|_{E}}
\]
and the dual norm $\|.\|_{G^{\ast}}$ can be defined similarly.

For the common vector norms with respect to the canonical scalar product in $ \R^{n} $, their dual norms  are  are given by :
\[
\|\cdot \|_{{1}\ast}=\|\cdot\|_{\infty}, \quad  \|\cdot \|_{{\infty}\ast}=\|\cdot\|_{1}\quad \mbox{and } \quad \|\cdot\|_{{2}\ast}=\|\cdot\|_{2}.
\]

For the matrix norms in $\R^{m\times n}$   with respect to the scalar product $ \langle A,B\rangle = \trace(A^\top B)$, where $\trace(A)$ is the trace of $A$
, we have  $ \|A\|_{F}\ast=\|A\|_{F}$ since  $\trace(A^\top A)=\|A\|_{F}^{2}$.

For the linear operator from $E$  to $G$, let $\|\cdot\|_{E,G}$  be the operator norm induced by
the norms  $\|\cdot\|_{E}$ and  $\|\cdot\|_{G}$. Consequently, for linear operators from $G$ to $E$, the norm
induced from the dual norms $\| \cdot \|_{E\ast}$ and
$\| \cdot \|_{G\ast}$, is denoted by $\| \cdot \|_{G\ast,E\ast}$.


We have the following  result for the adjoint operators and dual norms \cite{22.0}.
\begin{lemma}\label{lemma:adjoint}
$$
\|J\|_{E,G}=\|J^{\ast}\|_{G\ast,E\ast}.
$$
\end{lemma}


As mentioned in \cite{22.0}, it may be desriable  to compute $ \|J^{\ast}\|_{G\ast,E\ast} $ instead of $\|J\|_{E,G}$
when the dimension of the Euclidean space $G^{\ast}$  is  lower  than  $E$ because it implies a maximization over a space of smaller dimension.

Now, we consider a product space  $ E=E_{1}\times \cdots\times E_{p} $  where each Euclidean space  $E_{i}$ is equipped with the scalar product  $\langle\cdot,\cdot \rangle_{E_{i}}$ and the corresponding norm $\|\cdot\|_{E_{i}}$. In  $E$, we define the following scalar product
\[
\langle (x_{1},\cdots,x_{p}),(y_{1},\cdots,y_{p})\rangle=\langle x_{1},y_{1}\rangle_{E_{1}}+\cdots+\langle x_{p},y_{p}\rangle_{E_{p}},
\]
and the corresponding product norm
\[
\|(x_{1},\cdots,x_{p})\|_{v}=v(\|x_{1}\|_{E_{1}},\cdots,\|x_{p}\|_{E_{p}}),
\]
where   $v$  is an absolute norm on  $\R^{p}$, that is  $ v(|x|)=v(x)$, for any  $x\in \R^{p}$; see \cite{Higham2002Book} for details. We denote   $v_{\ast}$ is the dual norm of $v$
with respect to the canonical inner-product of $ \R^{p} $  and we are interested in determining the dual $ \|\cdot\|_{v\ast}$ of the
product norm $\|\cdot\|_{v}$  with respect to the scalar product of  $E$. The following result can be found in \cite{22.0}.

\begin{lemma}\label{lemmaProductNorm}
The dual of the product norm can be expressed by
\[
\|(x_{1},\cdots,x_{p})\|_{v\ast}=v(\|x_{1}\|_{E_{1\ast}},\cdots,\|x_{p}\|_{E_{p\ast}}).
\]
\end{lemma}


After introducing the necessary background in adjoint operators
and dual norms, we apply them to the condition numbers for ILS. We can view the Euclidean space $E$ with
norm $\| \cdot \|_E$ as the space of the input data in ILS and $G$ with norm $\| \cdot \|_G$ as the space
of the solution in ILS. Then
the function $g$ in (\ref{eq:g dfn}) is an operator from $E$
to $G$ and the condition number is the measurement of the
sensitivity of $g$ to the perturbation in its input data.

From \cite{Rice}, if $g$ is Fr{\'e}chet differentiable in neighborhood of  $y \in E$, then the absolute condition number of  $g$ at $y\in E$  is given by
\[
\kappa(y)=\|g'(y)\|_{E,G}=
\max_{\| z \|_E = 1} \| g'( y ) \cdot z \|_G ,
\]
where $ \|\cdot\|_{E,G} $  is the operator norm induced by the norms  $ \|\cdot\|_{E} $ and  $ \|\cdot\|_{G}$ and $g'( y )$ is the Fr{\'e}chet derivative of $g$ at $y$. If $ g(y) $ is nonzero, the \emph{relative condition number} of  $g$ at  $y\in E$ is defined as
\[
\kappa^{\rel}(y)=\kappa(y)\frac{\|y\|_{E}}{\|g'(y)\|_{G}}.
\]
The expression of  $\kappa(y)$  is related to the operator norm of the linear operator  $g'(y)$. Applying
Lemma~\ref{lemma:adjoint}, we have the following expression of
$\kappa(y)$ in terms of adjoint operator and dual norm:
\begin{equation} \label{eqnK}
 \kappa(y)=\max_{\|\Delta y\|_{E}=1} \|g'(y)\cdot \Delta y\|_{G}=\max_{\|z\|_{G\ast}=1}\|g'(y)^{\ast}\cdot z\|_{E\ast}.
\end{equation}

Now we consider the componentwise metric on a data space $ E=\R^{n}$. For any given $ y\in E$,  the subset $ E_{y} \in E$ is a set of all elements $ \Delta y\in E$ satisfying that $\Delta y_{i}=0$  whenever  $y_{i}=0$, $1\leq i\leq n$. Thus in a componentwise perturbation analysis, we measure the perturbation  $\Delta y\in E_{y}$ of  $y$  using the following componentwise norm with respect to $y$
\begin{equation}\label{eq:rr}
 \|\Delta y\|_{c}=\min\{\omega,|\Delta y_{i}|\leq \omega |y_{i}|,i=1,\ldots, n\}.
\end{equation}
Equivalently, it is easy to see that
\begin{equation}\label{eqnCNorm}
\|\Delta y\|_{c}=\max\left\{\frac{|\Delta y_{i}|}{|y_{i}|},y_{i}\neq 0\right\}=\left\|\left(\frac{|\Delta y_{i}|}{y_{i}}\right)\right\|_{\infty},
\end{equation}
where $y \in E_y$. 

In the following we consider the dual norm $\| \cdot \|_{c\ast}$
of the componentwise norm $\| \cdot \|_c$.
Let the product space $E$
be $\mathbb{R}^n$, each $E_i$ be $\mathbb{R}$,
and the absolute norm $v$ be $\| \cdot \|_{\infty}$.
Setting the norm $\| \Delta y_i \|_{E_i}$ in $E_i$ to
$| \Delta y_i | / | y_i |$ when $y_i \neq 0$, from
Definition~\ref{def:adj dual}, we have the dual norm
\[
\| \Delta y_i \|_{E_i\ast} =
\max_{z \neq 0} \frac{| \Delta y_i \cdot z |}{\| z \|_{E_i}} =
\max_{z \neq 0} \frac{| \Delta y_i \cdot z |}{|z| / |y_i|} =
|\Delta y_i| \, |y_i| .
\]
Applying Lemma~\ref{lemmaProductNorm} and (\ref{eqnCNorm}) and recalling
$\| \cdot \|_{\infty^*} = \| \cdot \|_1$, we derive the explicit expression of the dual norm
\begin{equation} \label{eqnDualC}
\| \Delta y \|_{c\ast} =
\| (\|\Delta y_1\|_{E\ast} ,..., \|\Delta y_n\|_{E\ast}) \|_{\infty\ast} =
\| (|\Delta y_1| \, |y_1| ,..., |\Delta y_n| \, |y_n|) \|_1 .
\end{equation}

Because of the condition $\| \Delta y \|_E = 1$ in the
condition number $\kappa(y)$ in (\ref{eqnK}), whether
$\Delta y$ is in $E_y$ or not, the expression
of the condition number $\kappa(y)$ remains valid. Indeed,
if $\Delta y \not\in E_y$, that is, $\Delta y_i\neq 0$
 while $y_i=0$ for some $i$, then $\| \Delta y \|_c = \infty$.
Consequently, such perturbation $\Delta y$ is
excluded from the calculation of $\kappa(y)$. Following
(\ref{eqnK}), we have the following lemma on the condition
number in adjoint operator and dual norm.

\begin{lemma} \label{lemmaK}
Using the above notations and the componentwise norm defined
in (\ref{eqnCNorm}), the condition number $\kappa(y)$ can be
expressed by
\[
\kappa(y) =
\max_{\| z \|_{G\ast} = 1}
\| (g'(y))^* \cdot z \|_{c\ast} ,
\]
where $\| \cdot \|_{c\ast}$ is given by (\ref{eqnDualC}).
\end{lemma}

After discussing the norms on the data space, in the next section,
we study the norms on the solution space, which can be either
componentwise or normwise. However, regardless of the norms chosen
in the solution space, we always use the componentwise norm in
the data space.

\subsection{Deriving condition number expressions via dual techniques}

Since $A^\top\Sigma_{pq}A$ is positive definite, the linear operator $g$ defined in \eqref{eq:g dfn} is continuously Fr{\'e}chet differentiable  in a neighborhood of the data $(A,\, b)$ and we denote by $J=g'(A,b)$  its derivative. For $ B\in \R^{m\times n}$ and  $c\in \R^{m}$, using the chain result of composition of derivatives, we get
\begin{align}\label{eqnJ}
J(B,c)&:=g'(A,b)\cdot (B,c)\nonumber \\
&=L^\top(A^\top\Sigma_{pq}A)^{-1}B^\top\Sigma_{pq}r-L^\top(A^\top\Sigma_{pq}A)^{-1}A^\top\Sigma_{pq}Bx+L^\top(A^\top\Sigma_{pq}A)^{-1}A^\top\Sigma_{pq}c\nonumber\\
&:=J_{1}(B)+J_{2}(c),
\end{align}
recalling that $r=b-Ax $ is the residual vector. Then $J(B,c)$ is a linear operator from the data space $\R^{m\times n}
\times \R^{m}$ to $\R^k$.
%

Using the definition of the adjoint operator and
the definition of the scalar
product in the data space $\mathbb{R}^{m \times n} \times \mathbb{R}^m$, an explicit
expression of the adjoint operator of the above $J(B, {c})$ is given in the following lemma.

\begin{lemma}\label{lemmaDualJ}
The adjoint of operator of the Fr{\'e}chet derivative $J(B,c)$ in \eqref{eqnJ} is given by
\begin{align*}
J^{\ast}&:\R^{k}\rightarrow \R^{m\times n}\times \R^{m}\\
 &u\mapsto \left(\Sigma_{pq}ru^\top L^\top(A^\top\Sigma_{pq}A)^{-1}-\Sigma_{pq}A(A^\top\Sigma_{pq}A)^{-1}Lux^\top,\Sigma_{pq}^\top A(A^\top\Sigma_{pq}A)^{-1}Lu\right).
\end{align*}
\end{lemma}

%
%

\begin{proof}
Using \eqref{eqnJ} and the definition of the scalar product in the matrix space, for any $u\in \R^{k}$, we have
\begin{align*}
\langle u,J_{1}(B)\rangle=&u^\top(L^\top(A^\top\Sigma_{pq}A)^{-1}B^\top\Sigma_{pq}r-L^\top(A^\top\Sigma_{pq}A)^{-1}A^\top\Sigma_{pq}Bx)\\
=&\trace(\Sigma_{pq}ru^\top L^\top(A^\top\Sigma_{pq}A)^{-1}B^\top)-\trace(xu^\top L^\top(A^\top\Sigma_{pq}A)^{-1}A^\top\Sigma_{pq}B)\\
=&\langle\Sigma_{pq}ru^\top L^\top(A^\top\Sigma_{pq}A)^{-1}, B\rangle-\langle\Sigma_{pq}A (A^\top\Sigma_{pq}A)^{-\top}Lux^\top,B\rangle\\
=&\langle \Sigma_{pq}ru^\top L^\top(A^\top\Sigma_{pq}A)^{-1}-\Sigma_{pq}A (A^\top\Sigma_{pq}A)^{-1}Lux^\top,B\rangle.
\end{align*}
For the second part of the adjoint of the derivative  $J$, we have
\begin{align*}
\langle u,J_{2}c\rangle
=&u^\top L^\top(A^\top\Sigma_{pq}A)^{-1}A^\top\Sigma_{pq}c=\langle \Sigma_{pq}A(A^\top\Sigma_{pq}A)^{-\top}Lu,c\rangle.
\end{align*}
Let
\[
J_1^*({u}) =\Sigma_{pq}ru^\top L^\top(A^\top\Sigma_{pq}A)^{-1}-\Sigma_{pq}A (A^\top\Sigma_{pq}A)^{-1}Lux^\top
\]
and
\[
J_2^*({u}) =
\Sigma_{pq}A(A^\top\Sigma_{pq}A)^{-1}Lu,
\]
then $\langle J^*({u}),\ (B, {c}) \rangle =
\langle (J_1^*({u}), J_2^*({u})) , \
(B, {c}) \rangle = \langle {u},\ J(B, {c}) \rangle$,
which completes the proof. \hfill $\Box$
\end{proof}

After obtaining an explicit expression of the adjoint operator
of the Fr{\'e}chet derivative,  we now give an explicit expression
of the condition number $\kappa$ (\ref{eqnK}) in terms
the dual norm in the solution space in the following theorem,
where $D_A$ denotes the diagonal matrix
$\mathrm{diag}(\vect(A))$.

\begin{theorem} \label{thmK}
The condition number for the  ILS
problem can be expressed by
\[
\kappa = \max_{\| {u} \|_{G\ast} = 1}
\| [ VD_A \ \
W D_{b} ]^{\top} L {u} \|_1 =
\| [VD_{A},WD_{b}]^\top L\|_{G\ast,1} ,
\]
where
\begin{align} \label{eqnV}
 V=&(A^\top\Sigma_{pq}A)^{-1}\otimes( \Sigma_{pq}r) ^\top-x^\top\otimes W,\quad W=(A^\top\Sigma_{pq}A)^{-1}A^\top\Sigma_{pq}.
\end{align}
\end{theorem}

\begin{proof}
Let $\Delta a_{ij}$  and  $\Delta b_{i}$ be the entries of $\Delta A $ and  $\Delta b$ respectively, using \eqref{eqnDualC}, we have
 \[
 \|(\Delta A,\Delta b)\|_{c\ast}=\sum_{i,j}|\Delta a_{ij}||a_{ij}|+\sum_{i}|\Delta b_{i}||b_{i}|.
 \]
Applying Lemma~\ref{lemmaDualJ}, we derive that
\begin{align*}
\|J^{\ast}(u)\|_{c\ast}=&\sum_{j=1}^{n}\sum_{i=1}^{m}  |a_{ij}||(\Sigma_{pq}ru^\top L^\top(A^\top\Sigma_{pq}A)^{-1}-\Sigma_{pq}A(A^\top\Sigma_{pq}A)^{-T}Lux^\top)_{ij}|\\+&\sum_{i=1}^{m}
|b_{i}||(\Sigma_{pq}A(A^\top\Sigma_{pq}A)^{-\top}Lu)_{i}|\\
=&\sum_{j=1}^{n}\sum_{i=1}^{m}|a_{ij}||((\Sigma_{pq}r)_{i}((A^\top\Sigma_{pq}A)^{-1}e_{j})^\top-x_{j}((A^\top\Sigma_{pq}A)^{-1}A^\top\Sigma_{pq} e_{i})^\top)Lu|\\
&+\sum_{i=1}^{m}
|b_{i}||((A^\top\Sigma_{pq}A)^{-1}A^\top\Sigma_{pq} e_{i})^\top Lu|,
\end{align*}
where $(\Sigma_{pq}r)_{i}$ is the $i$th component of $\Sigma_{pq}r$. Then it can be verified that $(\Sigma_{pq}r)_{i} (A^\top\Sigma_{pq}A)^{-1}e_{j}$
is the $((j-1)m+i)$th column of the $n \times (mn)$ matrix
$(A^\top\Sigma_{pq}A)^{-1}\otimes( \Sigma_{pq}r) ^\top$
and $x_j (A^\top\Sigma_{pq}A)^{-1}A^\top\Sigma_{pq} e_{i}$ is the $((j-1)m+i)$th
column of the $n \times (mn)$ matrix
$x^\top\otimes L^\top(A^\top\Sigma_{pq}A)^{-1}A^\top\Sigma_{pq}$ in
$V$ (\ref{eqnV}), implying that the above expression
equals
\[
\left\| \left[ \begin{array}{c}
D_A V^\top L {u} \\
D_{{b}} W^\top L {u}
\end{array} \right] \right\|_1 =
\left\| [ V D_A \ \ W D_{{b}} ]^{\top}
L {u} \right\|_1 .
\]
The theorem then follows from Lemma~\ref{lemmaK}. \hfill $\Box$
\end{proof}

The following case study discusses some commonly used norms
for the norm in the solution space to obtain some specific
expressions of the condition number $\kappa$.

\begin{corollary} \label{colKinfty1}
Using the above notations,
when the infinity norm is chosen as the norm in the solution
space $G$, we get
\begin{equation} \label{eqnKinfty1}
\kappa_{\infty} =
\left\| |L^{\top } V| \vect(|A|) +
|L^{\top}W| \, |{b}| \right\|_{\infty} .
\end{equation}
\end{corollary}

\begin{proof}
When $\| \cdot \|_G = \| \cdot \|_{\infty}$, the dual norm
$\| \cdot \|_{G\ast} = \| \cdot \|_1$. Thus
\begin{eqnarray*}
\kappa_{\infty} &=&
\left\| [ V D_A \ \ W D_{b} ]^\top
L \right\|_1 \\
&=& \| L^\top [ V D_A \ \ W D_{b} ]
\|_{\infty} \\
&=& \left\| |L^{\top} V| \vect(|A|) +
|L^{\top} W| \, |{b}| \right\|_{\infty} .
\qquad \Box
\end{eqnarray*}
\end{proof}

The following corollary gives an alternative expression of
$\kappa_{\infty}$.

\begin{corollary} \label{colKinfty2}
Using the above notations,
when the infinity norm is chosen as the norm in the solution
space $G$, we get
\begin{equation} \label{eqnKinfty2}
\kappa_{\infty} =
\left\|
\sum_{j=1}^n |L^{\top} (A^{\top} W A)^{-1}({e}_j (\Sigma_{pq}r)^{\top} -
x_j A^{\top}\Sigma_{pq})| \, |A(:,j)| +
|L^{\top} W| \, |{b}| \right\|_{\infty},
\end{equation}
where $e_j$ is $j$th column of $I_n$.
\end{corollary}

\begin{proof}
Partitioning
\[
V = [V_1 \ \cdots \ V_n] ,
\]
where each $V_j$, $1 \le j \le n$, is an $n \times m$ matrix,
we get
\begin{equation} \label{eqnKinfty2a}
\kappa_{\infty} =
\left\| |L^{\top}V| \vect(|A|) +
|L^{\top} W| \, |{b}| \right\|_{\infty} =
\left\| \sum_{j=1}^n |L^{\mathrm{T}}V_j| \, |A(:,j)| +
|L^{\top} W| \, |{b}| \right\|_{\infty} .
\end{equation}
Recalling that $(\Sigma_{pq}r)_{i} (A^\top\Sigma_{pq}A)^{-1}e_{j} -
x_j W e_{i}$ is the $((j-1)m+i)$th column
of $V$, we have
\[
V_j = (A^{\top}\Sigma_{pq}A)^{-1} ({e}_j (\Sigma_{pq}r)^{\top} -
x_j A^{\top}\Sigma_{pq}) .
\]
The expression (\ref{eqnKinfty2}) is obtained
by substituting $V_j$ in (\ref{eqnKinfty2a}) with the above expression
for $V_j$. \hfill $\Box$
\end{proof}
When the infinity norm is chosen as the norm in the solution
space $\R^n$, the corresponding relative mixed condition number is given by
\begin{equation}
\kappa_{\infty}^{\rel} =\frac{
\left\|
\sum_{j=1}^n |L^{\top} (A^{\top} W A)^{-1}({e}_j (\Sigma_{pq}r)^{\top} -
x_j A^{\top}\Sigma_{pq})| \, |A(:,j)| +
|L^{\top} W| \, |{b}| \right\|_{\infty}}{\|L^\top x\|_\infty}.
\end{equation}

In the following, we consider the 2-norm on the solution space and derive an upper bound for the corresponding condition number respect to the 2-norm on the solution space.

\begin{corollary} \label{colK2}
When the 2-norm is used in the solution space, we have
\begin{equation} \label{eqnK2}
\kappa_2 \le
\sqrt{k} \, \kappa_{\infty} .
\end{equation}
\end{corollary}

\begin{proof}
When $\| \cdot \|_G = \| \cdot \|_2$, then
$\| \cdot \|_{G\ast} = \| \cdot \|_2$. From Theorem~\ref{thmK},
\[
\kappa_2 = \| [ VD_A\ \
W D_{{b}}]^{\top} L \|_{2,1} .
\]
It follows from \cite{Higham2002Book} that for any matrix $B$,
$\| B \|_{2,1} = \max_{\| {u} \|_2 = 1} \| B {u} \|_1
= \| B \hat{{u}} \|_1$, where $\hat{{u}} \in \mathbb{R}^k$
is a unit 2-norm vector. Applying $\| \hat{{u}} \|_1 \le
\sqrt{k}\, \| \hat{{u}} \|_2$, we get
\[
\| B \|_{2,1} = \| B \hat{{u}} \|_1 \le
\| B \|_1 \| \hat{{u}} \|_1 \le
\sqrt{k}\, \| B \|_1 .
\]
Substituting the above $B$ with $[ VD_A\ \
W D_{b}]^{\top} L$, we have
\[
\kappa_2 \le \sqrt{k} \,
\| [ VD_A \ \ W D_{b}]^{\top} L \|_1 ,
\]
which implies (\ref{eqnK2}). \hfill $\Box$
\end{proof}

By now, we have considered the various mixed condition
numbers, that is, componentwise norm in the data space and
the infinity norm or 2-norm in the solution space. In the
rest of the subsection, we study the case of componentwise
condition number,
that is, componentwise norm in the solution space as well.

\begin{corollary} \label{colKc}
Considering the componentwise norm defined by
\begin{equation}\label{eq:comp norm}
\| {u} \|_c =
\min \{ \omega , \ | u_i | \le \omega \,
|(L^{\top} {x})_i|, \ i=1,...,k \} =
\max \{ |u_i| / |(L^{\top} {x})_i|, \ i=1,...,k \} ,
\end{equation}
in the solution space, we have the following three expressions
for the componentwise condition number
\begin{eqnarray*}
\kappa_c
&=& \| D_{L^{\top} {x}}^{-1}
L^{\top} [ VD_A \ \
W D_{{b}} ] \|_{\infty} \\
&=& \| |D_{L^{\top} {x}}^{-1} |
(|L^{\top} V| \vect (|A|) +
|L^{\top} W|\,|{b}| \|_{\infty} \\
&=& \left\|
\sum_{j=1}^n |D_{L^{\top} {x}}^{-1}
L^{\top} (A^{\top} \Sigma_{pq} A)^{-1}
({e}_j (\Sigma_{pq}r)^{\top} -
x_j A^{\top}\Sigma_{pq})| \, |A(:,j)| +
|D_{L^{\top} {x}}^{-1}
L^{\top} W | \, |{b}| \right\|_{\infty} .
\end{eqnarray*}
\end{corollary}

\begin{proof}
The expressions immediately follow from Theorem~\ref{thmK}
and Corollaries \ref{colKinfty1} and \ref{colKinfty2}. \hfill $\Box$
\end{proof}

\subsection{Condition estimations}
In this subsection, we will derive the tight upper bounds for $\kappa_\infty^{\rel}$ and $\kappa_c$, which can be estimated efficiently by the power method  \cite[Chapter 15]{Higham2002Book} during using the QR-Cholesky method  \cite{1.0} to solve ILS.

We first review QR-Cholesky method for solving ILS, which was proposed by Chandrasekaran, Gu, and Sayed in \cite{1.0}. The QR-Cholesky method is stable and efficient for solving ILS. For $A\in \R^{m\times n}$,  assume its QR decomposition
\[
 A=QR=\bordermatrix{&  \cr
                p&Q_1\cr
                q&Q_2\cr
                }R,
 \]
 where $Q\in \R^{m\times n}$ is orthogonal $R \in \R^ {n\times n}$ is upper triangular and nonsingular.
So
\[
A^{\top}\Sigma_{pq}A=R^{\top }(Q_{1}^{T}Q_{1}-Q_{2}^{\top }Q_{2})R.
\]
From the positive definiteness of  $A^{T}\Sigma_{pq}A $ and $R$ being nonsingular, the matrix $Q_{1}^{\top}Q_{1}-Q_{2}^{\top}Q_{2}$ is positive definite, which enable us to compute the Cholesky decomposition of $Q_{1}^{\top}Q_{1}-Q_{2}^{\top}Q_{2}=U^{\top}U $, where $U\in\R^{n\times n}$ is upper triangular and nonsingular. It is easy to see that
\[
A^{\top}\Sigma_{pq}A=R^{\top }U^{\top }UR.
\]
So the normal equation of ILS can be written
\[
(Q_{1}^{\top}Q_{1}-Q_{2}^{T}Q_{2})Rx=Q^{\top }\Sigma_{pq}b.
\]
Thus we can compute $x$ by solving the following system
$$
U^{\top }URx=Q^{\top }\Sigma_{pq}b,
$$
which can be implemented by  forward and backward substitutions for some right hands with different triangular coefficient matrices in sequence. Moreover, after the QR-Cholesky method \cite{1.0}, the matrices $R$ and $U$ have been computed, which can help us to reduce the computation cost when using the power method  \cite[Chapter 15]{Higham2002Book}  to estimate the upper bounds for $\kappa_\infty^{\rel}$ and $\kappa_c$.

In the following, we will give tight bounds for $\kappa_\infty^{\rel}$ and $\kappa_c$, which can be estimated efficiently by the power method  \cite[Chapter 15]{Higham2002Book}. Firstly, note that for any
matrix $B\in\R^{p\times q}$ and diagonal matrix
$D_v\in\R^{q\times q}$, 
\begin{equation*}\label{eq:norm diagonal}
  \|BD_v\|_\infty
  =\|\,|BD_v|\,\|_\infty
  =\|\,|B|\,|D_v|\,\|_\infty
  =\|\,|B||D_v|\bfe\,\|_\infty
  =\left\|\,\left|B\right||v|\right  \|_\infty.
\end{equation*}
where ${\bf e}=[1,\,\ldots,1]^\rt \in \R^q$. With the above property and triangle inequality, we can prove the following theorem and its proof is omitted.

\begin{theorem}
  With the notations above, we have the following bounds
  \begin{align*}
   \frac{1}{2}\kappa_\infty^\u \leq &\kappa_\infty^{\rel}\leq \kappa_\infty^\u,\quad \frac{1}{2}\kappa_c^\u \leq \kappa_c\leq \kappa_c^\u,
  \end{align*}
  where
  \begin{align*}
  \kappa_\infty^\u&=\frac{\left\|L^\top R^{-1}U^{-1}U^{-\top}R^{-\top}\left[{e}_1 (\Sigma_{pq}r)^{\top} -
x_1 A^{\top}\Sigma_{pq}, \, \ldots, {e}_n (\Sigma_{pq}r)^{\top} -
x_n A^{\top}\Sigma_{pq}\right]D_A\right\|_\infty}{\left\|L^\top x\right\|_\infty} \\
&+\frac{\left\|L^\top R^{-1}U^{-1}U^{-\top}R^{-\top}A^\top\Sigma_{pq} D_b\right\|_\infty}{\left\|L^\top x\right\|_\infty},\\
  \kappa_c^\u&=\left\|D_{L^{\top} {x}}^{-1}L^\top R^{-1}U^{-1}U^{-\top}R^{-\top} \left[{e}_1 (\Sigma_{pq}r)^{\top} -
x_1 A^{\top}\Sigma_{pq}, \, \ldots, {e}_n (\Sigma_{pq}r)^{\top} -
x_n A^{\top}\Sigma_{pq}\right]D_A\right\|_\infty\\
&+\left\|D_{L^{\top} {x}}^{-1}L^\top R^{-1}U^{-1}U^{-\top}R^{-\top} A^\top\Sigma_{pq} D_b\right\|_\infty.
  \end{align*}
\end{theorem}

The power method \cite[Chapter 15]{Higham2002Book} is an efficient algorithm to estimate 1-norm for any matrix $B$. The main computational cost of the power method is matrix-vector multiplications $Bv$ and $B^\top v$ for some vector $v$. In view of the expressions of $\kappa_\infty^\u$ and $\kappa_c^\u $, when the power method is implemented, the computation complexity mainly relies on $R^{-1}U^{-1}U^{-\top}R^{-\top}v$ for some vector $v$, which can be computed by the forward and backward substitutions for some
linear systems with triangular coefficient matrices. Because the upper triangular matrices $R$ and $U$ have been computed,  the computational cost of the power method to estimate $\kappa_\infty^\u$ and $\kappa_c^\u $ can be reduced from the fact of the triangular structures of $R$ and $U$. Thus the power methods for $\kappa_\infty^\u$ and $\kappa_c^\u $ are efficient and their detailed descriptions are omitted.

\section{Numerical examples}\label{sec:nume ex}

In this section we do some numerical examples to show the effectiveness of the previous derived results. All the
computations are carried out using \textsc{Matlab} 8.1 with the machine precision
$\mu=2.2 \times 10^{-16}$.

In this section, we fix $m=16,\, n=8$ and $p=10$. We adopt the matrix  as generated in \cite{5.0}.   The coefficient matrix $A$ has the following form:
$$
A=\bordermatrix{&  \cr
                p&S_1DV\cr
                q&\frac{1}{2}S_2DV\cr
                }\in \R^{m\times n},
$$
where $V\in \R^{n\times n},\, S_1\in\R^{p\times n}$ and $S_2\in \R^{q\times n}$ are randomized orthogonal matrices by using QR decomposition of randomized matrices, and $D\in \R^{n\times n}$ is diagonal with diagonal elements distributed exponentially from $\delta^{-1}$ to 1. It can be verified that $A^\top \Sigma_{pq}A=(3/4)V^\top D^2V$, so the uniqueness condition \eqref{eq:unique cond} is satisfied.  Let $v$ be a $n\times 1$ vector with $v_1=v_2=\epsilon$ and $v_n=1/\epsilon$, and other components are set to zero. Then the vector $b$ is constructed as follows
\begin{equation}\label{eq:b}
b=A\cdot v+10^{-5}\cdot b_2,
\end{equation}
where $b_2$ is an unitary vector satisfying $A^\top b_2=0$. We use  the QR-Cholesky method \cite{1.0} for solving ILS to compute the solution $x$. Usually the solution $x$ have badly scaled components,  for example, the first and second components of $x$ often have the same order of $\epsilon$ while the last component of $x$ is order of $1/\epsilon$. For the perturbations $E$ on $A$ and $f$ on $b$, we generate them as
\begin{equation}\label{eq:pert}
\Delta A=\varepsilon \cdot \Delta A_{1}\odot A,\quad, \Delta b=\varepsilon \cdot \Delta b_{1}\odot  b,
\end{equation}
where $\varepsilon=10^{-10}$,  each components of $\Delta A_1 \in \R^{m\times n}$ and $\Delta b_1\in \R^{m}$ are uniformly distributed in the interval $ (-1,1) $,  and $\odot$ denotes the componentwise multiplication of two conformal dimensional matrices. Let the perturbed solution $\tilde x$ is computed by \textsc{Matlab} using  the QR-Cholesky method \cite{1.0} for the following perturbed ILS problem
\begin{equation}\label{eq:pert ILS}
\min\left((b+\Delta b)-(A+\Delta A)\tilde{x}\right)^{\top}\Sigma_{pq}\left((b+\Delta b)-(A+\Delta A)\tilde x\right).
\end{equation}

For the $L$ matrix in our condition numbers, we choose
\[
L_0=I_n,\quad
L_1=\left(\begin{matrix} 1 & 0 \\
                          0 & 1 \\
                          0 & 0\\
                          \vdots&\vdots\\
                          0&0\end{matrix}\right)\in \R^{n\times 2},\quad
L_2=\left(\begin{matrix} 0 & 0 &\cdots& 1\end{matrix}\right)^{\mathrm{\top}} \in \R^{n\times 1}.
\]
Thus, corresponding to the above three matrices, the whole
${x}$, the subvector $[x_1 \ x_2]^{\mathrm{T}}$, and
the component $x_n$ are selected respectively.

We measure the normwise, mixed and componentwise relative errors in
$L {x}$ defined by
\[
r_2^{\rm rel} =
\frac{\|L^{\top} \tilde{{x}} -
L^{\top} {x} \|_2}
{\|L^{\top} {x} \|_2},\quad
r_\infty^{\rm rel} =
\frac{\|L^{\top} \tilde{{x}} -
L^{\top} {x} \|_\infty}
{\|L^{\top} {x} \|_\infty},\quad
r_c^{\rm rel} =
\frac{\|L^{\top} \tilde{{x}} -
L^{\top} {x} \|_c}
{\|L^{\top} {x} \|_c},
\]
where $\|\cdot\|_c$ is the componentwise norm defined
in \eqref{eq:comp norm}. The relative normwise condition number $\alpha_1$ given in \cite{5.0} for the whole vector $x$ is defined as
$$
\alpha_1=\frac{\|M^{-1}A^{\top}\|_{2}\|b\|_{2}+\|(x^{\top}\otimes M^{-1}A^\top\Sigma_{pq})-(r^{T}\Sigma_{pq}\otimes M^{-1}\Pi)\|_{2}\|A\|_{F}}{\|x\|_{2}}.
$$
where $M=A^{\top }\Sigma_{pq}A$  and $\Pi$ is the vec-permutation matrix satisfying $\Pi \,\vect(B)=\vect(B^\top)$. The relative normwise condition number $\alpha_2$  for a linear function of the solution $x$ defined in \eqref{eq:g dfn} is derived in \cite[Eqn.(3.8)]{LiWang2016}
$$
\alpha_2=\frac{\left\|L^\top M^{-1}\left[\|A\|_F \|r\|_2(I_n-\frac{1}{\|r\|_2^2}A^\top r x^\top)\,\, -\|b\|_2  A^\top \,\, \|A\|_F \|x\|_2A^\top(I_m-\frac{1}{\|r\|_2^2}rr^\top)\right]\right\|_2}{\|L^\top x\|_2}.
$$
So the above normwise condition number is a relative condition number corresponding to the following definition
\begin{equation*} 
\alpha_2 = \lim_{\epsilon \to 0}
\sup_{\left \|\left [\Delta A]/\|A\|_F\,\, \Delta b/\|b\|_2\right] \right\|_F \le \epsilon }
\frac{\| g (A + \delta A, \, b+\Delta b) - g (A,\, b) \|_2 }{\epsilon\, \| g (A,\, b) \|_2}.
\end{equation*}

In Table \ref{tab:T1}, we  test different choices of $\epsilon$ and $\delta$. It is observed that when $\epsilon$ (or $\delta$) deceases to 0, ${\rm cond}(A^{\top}\Sigma_{pq}A)$ increases which means that ILS tends to be more ill-conditioned. The first order perturbation bounds: $\alpha_2\cdot 10^{-10}$, $\kappa_\infty^{\rel} \cdot 10^{-10}$ and $\kappa_c^{\rel} \cdot 10^{-10}$ can bound the true relative errors   $r_2^\rel$, $r_\infty^\rel$ and $r_c^\rel$ for different $L$. On the other hand,  the first order perturbation bounds $\alpha_1\cdot 10^{-10}$ cannot bound the true relative normwise errors when $L=L_1$ or $L=L_2$ because $\alpha_1$ does not take account of the conditioning of the particular component of $x$.  Also, the first and second component of the solution $x$ are ill-conditioned, which is coincide with numerical values of $\alpha_2$, $\kappa_\infty^{\rel} $ and $\kappa_c^{\rel}$ for $L_1$. While the last component of $x$ has better conditioning, which show the effectiveness of $\kappa_\infty^{\rel}$ and $\kappa_c^{\rel}$. Moreover, the values of $\alpha_2$ are always larger than the counterparts of
$\kappa_\infty^{\rel}$ and $\kappa_c^{\rel}$ for most cases, which means that it reasonable to consider the componentwise perturbations on the data instead of the normwise perturbations.

\begin{landscape}

\begin{table}
\caption{
Comparison of condition numbers with the corresponding relative
errors and the \textsc{Matlab}  ${\rm cond}(A^{\top}\Sigma_{pq}A)$.
} \label{tab:T1}
\centering
\begin{tabular}{ccccccccccc}
\hline
$\epsilon$ & $\delta$ & $L$ &${\rm cond}(A^{\top}\Sigma_{pq}A)$& $r_2^\rel$ & $\alpha_1 $& $\alpha_2$ & $r_\infty^{\rel}$&$\kappa_\infty^{\rel}$ & $r_c^{\rel}$&$\kappa_c^{\rel}$  \\
\hline
$10^{-3}$ & $10^{-3}$ &
$I$&  1.0431e+06& 1.0066e-08 & 2.1894e+03 & 1.7760e+03 & 6.6497e-09 & 4.1222e+02 & 9.6157e-01 & 4.2462e+10  \\
 & &  $L_1$& & 2.8307e-03 & & 7.0546e+08 & 3.4033e-03 & 3.3942e+08 & 4.2984e-03 & 4.2869e+08\\
 & & $L_2$ & &1.3381e-09 & & 3.3965e+02 & 6.6497e-09 & 8.8258e+01 & 1.3381e-09 & 8.8258e+01  \\
\hline
$10^{-3}$ & $10^{-6}$  &
$I$ &   1.0347e+12 &3.6836e-05 & 2.7311e+06 & 1.9411e+06 & 2.1181e-05 & 1.1591e+06 & 5.9967e+00 & 2.4897e+11  \\
&&$L_1$&& 2.3777e-02 & & 1.3607e+09 & 2.3836e-02 & 1.3044e+09 & 2.3836e-02 & 1.3044e+09\\
&&$L_2$ & &7.8792e-06 & & 4.5249e+05 & 2.1181e-05 & 4.3230e+05 & 7.8792e-06 & 4.3230e+05 \\
\hline
$10^{-6}$ & $10^{-3}$ &
$I$&  1.0563e+06&1.5060e-09 & 1.7958e+03 & 1.5477e+03 & 1.0985e-09 & 2.7511e+02 & 1.4888e+00 & 3.3108e+11\\
&&$L_1$& &6.1005e-01 & & 8.9987e+11 & 6.8260e-01 & 2.5527e+11 & 1.0943e+00 & 3.1387e+11 \\
&&$L_2$&  &3.4119e-11 & & 1.9931e+02 & 1.0985e-09 & 3.5591e+01 & 3.4119e-11 & 3.5591e+01  \\
\hline
$10^{-6}$ & $10^{-6}$ &
$I$&   9.9265e+11&3.9831e-06 & 1.7151e+06 & 1.5576e+06 & 2.2579e-06 & 1.0670e+05 & 1.6437e+01 & 7.0924e+11 \\
&&$L_1$& & 3.4589e+00 & & 1.4342e+12 & 3.7110e+00 & 1.7458e+11 & 3.7110e+00 & 1.7458e+11\\
&&$L_2$ & &1.9323e-06 & & 7.9804e+05 & 2.2579e-06 & 9.2643e+04 & 1.9323e-06 & 9.2643e+04  \\
\hline
\end{tabular}
\end{table}

\end{landscape}


In the rest of this section, we will test the effectiveness of the linearization estimate $\bar \mu$ for the normwise backward error $\mu$ of ILS in Section \ref{sec:backward error}. The coefficient matrix $A$ is generated as previous. For the vector $b$, we adopt two ways to generate it. One is as the previous method \eqref{eq:b} with $\epsilon=10^{-3}$. Another way is to generate $b$ by using \textsc{Matlab} command {\sf randn}, i.e. $b$ is a standard Gaussian vector.  For given $A$ and $b$, we generate perturbations $\Delta A$ and $\Delta b$ as in \eqref{eq:pert}. Let the computed solution $y$ is calculated by the QR-Cholesky method \cite{1.0} for the perturbed ILS problem \eqref{eq:pert ILS}.
For the computed solution $y$, its normwise backward error $\mu$  is defined by \eqref{eq:xhwu}, and its linearization estimate $\bar \mu$ for the normwise backward error $\mu$ is given by \eqref{eq:xhwugj}. We always use the common choice  $\theta=1$ in \eqref{eq:xhwu}.

Please note that there is no explicit expression for the normwise backward error $\mu$. Since the perturbations $\Delta A$ and $\Delta b$ are known in advance, we can calculate the following quantity $\mu_1$ to approximate $\mu$:
$$
 \mu_{1}= \|[\Delta A, \theta \Delta b]\|_F,
 $$
 and compare $\mu_1$ with the linearization estimate $\bar \mu$ to show the effectiveness of $\bar \mu$. From the definition of the normwise backward error $\mu$ defined in \eqref{eq:xhwu}, it is easy to see that $\mu \leq \mu_1$. Note that $\mu$ may be much smaller that $\mu_1$ because $\mu$ is the smallest perturbation magnitude over the set of all perturbations $\xi_{\ILS}$ \eqref{eq:xi}. We test different choices of the perturbations magnitude $\varepsilon$ in \eqref{eq:pert} and the parameter $\delta$ which determines the conditioning of the coefficient matrix $A^\top \Sigma_{pq}A$ in the normal equation \eqref{eq:normal}. We report the numerical values of $\mu_1$, $\bar \mu$ and the residual norm $\gamma$, where
$$
\gamma=\left\|(A+\Delta A)^{\top}\Sigma_{pq}[(b+\Delta b)-(A+\Delta A)y]\right\|_2.
$$

From Tables \ref{tab:b1} and \ref{tab:b2}, it is observed that the residual norm $\gamma$ increases when $\delta$ decreases from $10^{-1}$ to $10^{-8}$ under both small and large perturbations, which means that when ILS becomes more ill-conditioned the computed solution cannot be calculated accurately. Also the linearization estimate $\bar \mu$ increases corresponding to the decrease  of $\delta$ because more smaller $\delta$ makes more ill-posedness of ILS. For the vector $b$ generated by \eqref{eq:b}, $\bar \mu$ can be much bigger than $\mu_1$ when large perturbation magnitude $\varepsilon=10^{-7}$ both for well-conditioned ($\delta=10^{-1}$) and ill-conditioned ($\delta =10^{-8}$) ILS problem, while we cannot conclude that $\bar \mu$ gives a bad estimation for $\mu$ because $\mu$ is the smallest perturbation magnitude to let the computed solution $y$ be the exact solution of the perturbed ILS \eqref{eq:pert ILS} mathematically. In Table \ref{tab:b1}, when the perturbation magnitude $\varepsilon=10^{-14}$, $\bar \mu$ can approximate $\mu_1$ more closely compared the cases for $\varepsilon=10^{-7}$. Especially, when $\delta =10^{-8}$, $\bar \mu$ and $\mu_1$ have the same order, which means that the linearization estimate $\mu_1$ can approximate the normwsie backward error $\mu$ accurately. For the standard Gaussian vector $b$, $\bar \mu$ can approximate $\mu_1$ more accurately for well-conditioned and ill-conditioned ILS problem under both small and large perturbations. The differences between $\bar \mu$ and $\mu_1$ are up to  a hundredfold. Overall, the linearization estimate $\bar \mu$ is effective.

\section{Concluding Remarks}\label{sec:con}

In this paper we studied the perturbation analysis for the indefinite least squares problem. The linearization estimate for the normwise backward error was obtained and condition number expressions for the linear function of the ILS solution were derived through the dual techniques under componentwise perturbations for the input data. Moreover,  these condition numbers could be estimated efficiently by the power method \cite[Chapter 15]{Higham2002Book} when solving ILS using the QR-Cholesky method \cite{1.0}. Numerical examples validated the effectiveness of the proposed condition numbers and  linearization estimate for the normwise backward error.




\begin{table}[!htpb]
  \caption{Comparisons between $\mu_1$ and $\bar\mu$, where the vector $b$ is generated by \eqref{eq:b}}
  \label{tab:b1}
  \centering
   \begin{tabular}{ccccc}
   \hline
     $\varepsilon$& $ \delta $ & $ \gamma$  & $ \mu_{1} $ & $ \bar{\mu} $\\
      \hline
    $10^{-7}$ & $ 10^{-1} $ &  1.6492e-13 & 3.1340e-05 & 2.4542e-15     \\
      & $ 10^{-4} $ &    3.9891e-12 & 2.1504e-05 & 5.0464e-14     \\
       & $ 10^{-8} $ &    1.2220e-08 & 6.3091e-06 & 1.2096e-09       \\
     \hline
      $10^{-14}$ & $ 10^{-1} $ & 7.1711e-14 & 3.7939e-12 & 2.0983e-16
     \\
       & $ 10^{-4} $ &  3.0805e-11 & 2.5107e-12 & 3.7992e-14    \\
        & $ 10^{-8} $ &   5.9456e-09 & 1.2178e-12 & 7.9989e-12    \\
      \hline
    \end{tabular}
\end{table}

\begin{table}[!htpb]
  \caption{Comparisons between $\mu_1$ and $\bar\mu$, where the vector $b$ is generated by \textsc{Matlab} command {\sf randn} }
  \label{tab:b2}
  \centering
   \begin{tabular}{ccccc}
   \hline
     $\varepsilon$& $ \delta $ & $ \gamma$  & $ \mu_{1} $ & $ \bar{\mu} $\\
      \hline
    $10^{-7}$ & $ 10^{-1} $ &   7.5362e-16 & 3.3461e-07 & 1.0935e-08   \\
      & $ 10^{-4} $ &     4.0348e-13 & 1.8978e-07 & 9.1348e-09    \\
       & $ 10^{-8} $ &    2.9291e-08 & 1.7096e-07 & 4.0896e-09       \\
     \hline
      $10^{-14}$ & $ 10^{-1} $ &   8.7788e-16 & 3.4058e-14 & 2.1264e-15      \\
       & $ 10^{-4} $ &  1.5979e-12 & 2.1033e-14 & 6.3998e-16   \\
        & $ 10^{-8} $ &   6.2215e-11 & 1.9163e-14 & 8.1922e-16    \\
      \hline
    \end{tabular}
\end{table}


\begin{thebibliography}{99}

\bibitem{22.0}
Marc Baboulin and Serge Gratton.
\newblock Using dual techniques to derive componentwise and mixed condition
  numbers for a linear function of a linear least squares solution.
\newblock {\em BIT}, 49(1):3--19, 2009.

\bibitem{5.0}
Adam Bojanczyk, Nicholas~J. Higham, and Harikrishna Patel.
\newblock Solving the indefinite least squares problem by hyperbolic {QR}
  factorization.
\newblock {\em SIAM J. Matrix Anal. Appl.}, 24(4):914--931 (electronic), 2003.

\bibitem{Cucker2013Book}
Peter B{\"u}rgisser and Felipe Cucker.
\newblock {\em Condition: {T}he geometry of numerical algorithms}.
\newblock Springer, Heidelberg, 2013.

\bibitem{1.0}
S.~Chandrasekaran, M.~Gu, and A.~H. Sayed.
\newblock A stable and efficient algorithm for the indefinite linear
  least-squares problem.
\newblock {\em SIAM J. Matrix Anal. Appl.}, 20(2):354--362, 1999.

\bibitem{Chang}
X.-W. Chang and D.~Titley-Peloquin.
\newblock Backward perturbation analysis for scaled total least-squares
  problems.
\newblock {\em Numer. Linear Algebra Appl.}, 16(8):627--648, 2009.

\bibitem{Cox}
Anthony~J. Cox and Nicholas~J. Higham.
\newblock Backward error bounds for constrained least squares problems.
\newblock {\em BIT}, 39(2):210--227, 1999.

\bibitem{CuckerDiaoWei2007}
Felipe Cucker, Huaian Diao, and Yimin Wei.
\newblock On mixed and componentwise condition numbers for {M}oore-{P}enrose
  inverse and linear least squares problems.
\newblock {\em Math. Comp.}, 76(258):947--963, 2007.

\bibitem{Demmel1987NumerMath}
James~W. Demmel.
\newblock On condition numbers and the distance to the nearest ill-posed
  problem.
\newblock {\em Numer. Math.}, 51(3):251--289, 1987.

\bibitem{Gohberg}
I.~Gohberg and I.~Koltracht.
\newblock Mixed, componentwise, and structured condition numbers.
\newblock {\em SIAM J. Matrix Anal. Appl.}, 14(3):688--704, 1993.

\bibitem{Graham81}
Alexander Graham.
\newblock {\em Kronecker Products and Matrix Calculus: With Applications}.
\newblock Halsted Press, 1981.

\bibitem{Grcar2011}
J.F. Grcar.
\newblock Unattainable of a perturbation bound for indefinite least squares
  problem.
\newblock {\em arXiv:1004.4921v5}, 2011.

\bibitem{Grcar03optimalsensitivity}
Joseph~F. Grcar.
\newblock Optimal sensitivity analysis of linear least squares.
\newblock Technical report, 2003.

\bibitem{Grcar07estimatesof}
Joseph~F. Grcar, Michael~A. Saunders, and Zheng Su.
\newblock Estimates of optimal backward perturbations for linear least squares
  problems. \newblock Technical report,  2007.

\bibitem{Hassibi96linearestimation}
Babak Hassibi, Ali~H. Sayed, and Thomas Kailath.
\newblock Linear estimation in {K}rein spaces - {P}art {I}: {T}heory.
\newblock {\em IEEE Transactions on Autmatic Control}, 3(2):18--33, 1996.

\bibitem{HighamHigham1992BackCondStr}
Desmond~J. Higham and Nicholas~J. Higham.
\newblock Backward error and condition of structured linear systems.
\newblock {\em SIAM J. Matrix Anal. Appl.}, 13(1):162--175, 1992.

\bibitem{Higham}
Nicholas~J. Higham.
\newblock Computing error bounds for regression problems.
\newblock In {\em Statistical analysis of measurement error models and
  applications ({A}rcata, {CA}, 1989)}, volume 112 of {\em Contemp. Math.},
  pages 195--208. Amer. Math. Soc., Providence, RI, 1990.

\bibitem{MR1314843}
Nicholas~J. Higham.
\newblock A survey of componentwise perturbation theory in numerical linear
  algebra.
\newblock In {\em Mathematics of {C}omputation 1943--1993: a half-century of
  computational mathematics ({V}ancouver, {BC}, 1993)}, volume~48 of {\em Proc.
  Sympos. Appl. Math.}, pages 49--77. Amer. Math. Soc., Providence, RI, 1994.

\bibitem{Higham2002Book}
Nicholas~J. Higham.
\newblock {\em Accuracy and stability of numerical algorithms}.
\newblock SIAM, Philadelphia, PA, second edition, 2002.

\bibitem{Karlson}
Rune Karlson and Bertil Wald{\'e}n.
\newblock Estimation of optimal backward perturbation bounds for the linear
  least squares problem.
\newblock {\em BIT}, 37(4):862--869, 1997.

\bibitem{LiWang2016}
Hanyu Li and Shaoxin Wang.
\newblock On the partial condition numbers for the indefinite least squares
  problem.
\newblock {\em arXiv:1605.05164v1}, 2016.

\bibitem{LiWangYang2014}
Hanyu Li, Shaoxin Wang, and Hu~Yang.
\newblock On mixed and componentwise condition numbers for indefinite least
  squares problem.
\newblock {\em Linear Algebra and its Applications}, 448:104 -- 129, 2014.

\bibitem{Liu2011}
Qiaohua Liu and Xianjuan Li.
\newblock Preconditioned conjugate gradient methods for the solution of
  indefinite least squares problems.
\newblock {\em Calcolo}, 48(3):261--271, 2011.

\bibitem{LiuLiuCalcolo2014}
Qiaohua Liu and Aijing Liu.
\newblock Block {SOR} methods for the solution of indefinite least squares
  problems.
\newblock {\em Calcolo}, 51(3):367--379, 2014.

\bibitem{LiuZhang2013ILS}
Qiaohua Liu and Fudao Zhang.
\newblock Incomplete hyperbolic {G}ram-{S}chmidt-based preconditioners for the
  solution of large indefinite least squares problems.
\newblock {\em J. Comput. Appl. Math.}, 250:210--216, 2013.

\bibitem{LiuXinguoNLAA2012}
Xin-Guo Liu and Na~Zhao.
\newblock Linearization estimates of the backward errors for least squares
  problems.
\newblock {\em Numer. Linear Algebra Appl.}, 19(6):954--969, 2012.

\bibitem{Malyshev}
A.~N. Malyshev.
\newblock Optimal backward perturbation bounds for the {LSS} problem.
\newblock {\em BIT}, 41(2):430--432, 2001.

\bibitem{MastronardiVanDooren2014BIT}
Nicola Mastronardi and Paul Van~Dooren.
\newblock An algorithm for solving the indefinite least squares problem with
  equality constraints.
\newblock {\em BIT}, 54(1):201--218, 2014.

\bibitem{Rice}
John~R. Rice.
\newblock A theory of condition.
\newblock {\em SIAM J. Numer. Anal.}, 3:287--310, 1966.

\bibitem{Sayed96inertiaproperties}
Ali~H. Sayed, Babak Hassibi, and Thomas Kailath.
\newblock Inertia properties of indefinite quadratic forms.
\newblock {\em IEEE Signal Process. Lett}, 3(2):57--59, 1996.

\bibitem{Skeel79}
Robert~D. Skeel.
\newblock Scaling for numerical stability in {Gaussian} elimination.
\newblock {\em J. ACM}, 26(3):494--526, 1979.

\bibitem{Stewart}
G.~W. Stewart.
\newblock Research, development, and {LINPACK}.
\newblock In {\em Mathematical software, {III} ({P}roc. {S}ympos., {M}ath.
  {R}es. {C}enter, {U}niv. {W}isconsin, {M}adison, {W}is., 1977)}, pages 1--14.
  Publ. Math. Res. Center, No. 39. Academic Press, New York, 1977.

\bibitem{30.0}
Sabine Van~Huffel and Joos Vandewalle.
\newblock {\em The total least squares problem:Computational aspects and analysis}.
\newblock SIAM, Philadelphia,
  PA, 1991.


\bibitem{Walden}
Bertil Wald{\'e}n, Rune Karlson, and Ji~Guang Sun.
\newblock Optimal backward perturbation bounds for the linear least squares
  problem.
\newblock {\em Numer. Linear Algebra Appl.}, 2(3):271--286, 1995.

\bibitem{Wang2009ILS}
Qian Wang.
\newblock Perturbation analysis for generalized indefinite least squares
  problems.
\newblock {\em J. East China Norm. Univ. Natur. Sci. Ed.}, (4):47--53, 2009.

\bibitem{Xu2004}
Hongguo Xu.
\newblock A backward stable hyperbolic {QR} factorization method for solving
  indefinite least squares problem.
\newblock {\em J. of Shanghai University (English Edition)}, 8(4):391--396,
  2004.

\bibitem{Zhou2015}
Tong-Yu Zhou.
\newblock On the condition number and backward error analysis for indefinite
  linear least squares problems.
\newblock Master's thesis, Northeast Normal University, May 2015.
\newblock In Chinese.

\end{thebibliography}
\end{document}